\documentclass[11pt]{amsart}

\title{Log-concavity of the partition function}
\author[Stephen DeSalvo and Igor Pak]{Stephen DeSalvo$^\ast$ and Igor Pak$^\ast$}

\thanks{\thinspace ${\hspace{-.45ex}}^\ast$Department of Mathematics, UCLA, Los Angeles, CA 90095, USA; \ts
\texttt{\{stephendesalvo;pak\}@math.ucla.edu}}
\date{July 4, 2014}

\usepackage{amsmath, amsthm, amsfonts, amssymb}

\usepackage[all]{xy}
\usepackage[margin=1in]{geometry}
\usepackage{graphicx}
\usepackage{subfigure}

\newtheorem{thm}{Theorem}[section]

\newtheorem{cor}[thm]{Corollary}
\newtheorem{prop}[thm]{Proposition}

\newtheorem{lemma}[thm]{Lemma}
\newtheorem{conjecture}[thm]{Conjecture}
\newtheorem{con}[thm]{Conjecture}

\newcommand{\ignore}[1]{}

\begin{document}

\def\BI{\mathbb{I}}

\def\e{\mathbb{E \,}}
\def\M{{\mathcal M}}
\def\BN{\mathbb{N}}
\def\BR{\mathbb{R}}
\def\BZ{\mathbb{Z}}
\def\L{{\mathcal L}}
\def\p{\mathbb{P}}
\def\var{{\rm Var}}
\def\cov{{\rm Cov}}
\def\diam{{\rm diam\, }}
\def\SD{{\rm SD}}
\def\A{{\mathcal A}}
\def\B{{\mathcal B}}
\def\C{{\mathcal C}}
\def\F{{\mathcal F}}
\def\X{{\mathcal X}}
\def\BX{{\mathbb{X}}}
\def\BY{{\mathbb{Y}}}
\def\BI{{\mathbb{I}}}

\def\.{\hskip.06cm}
\def\ts{\hskip.03cm}
\def\mts{\hspace{-.04cm}}

\begin{abstract}
We prove that the partition function $p(n)$ is log-concave for all~$n>25$.
We then extend the results to resolve two related conjectures by Chen
and one by Sun.  The proofs are based on Lehmer's estimates on the 
remainders of the Hardy--Ramanujan and the Rademacher series for~$p(n)$.


\smallskip
\noindent \textbf{MSC classes:} 05A17, 11N37, 65G99

\end{abstract}

\maketitle

\section{Introduction}

\noindent
A sequence $\{a_n\}$ is \emph{log-concave} if it satisfies
\[a_{n}^2 \. - \. a_{n-1}\ts a_{n+1} \. \ge \. 0 \quad \text{for all $n$\ts.}\]
Notable examples of log-concave sequences are the binomial coefficients,
the Stirling numbers, the Bessel numbers, etc.~(see~\cite{B,LW,S} for
well-written surveys).
Despite the interest in log-concave sequences and the detailed asymptotics
of the partition function in combinatorics, the log-concavity of $p(n)$
for all $n>25$ remained an open problem 
(see Section~\ref{sec:final}).  Note that it fails for all odd values $n \leq 25$.

\begin{thm}\label{log concave theorem}
Sequence~$p(n)$ is log-concave for all $n>25$.
\end{thm}

This intuition behind the theorem comes from explicit calculations for small~$n$,
and the Hardy--Ramanujan asymptotic formula \cite{HR}
\begin{equation} \label{HR asymptotics}
 p(n) \. \sim \. \frac{1}{4\sqrt{3}\ts n} \,\, e^{\pi \sqrt{\frac{2}{3}\ts n}} \, \ \quad \text{as } \ n\to\infty\..
 \end{equation}
The function on the r.h.s.~of~(\ref{HR asymptotics}) is easily shown to be log-concave,
but without guaranteed error bounds there is no way of knowing precisely when the
asymptotic formula dominates the calculation.
Our proof is based on the estimates by Lehmer~\cite{L1,L2} which provide improved explicit guaranteed
error estimates valid for all~$n$. We use these to prove the theorem for all $n\geq 2600$,
and check log-concavity for smaller values using {\sc Mathematica}.

\smallskip

We turn to the following two recent conjectures by William Chen which partially motivated our
study, see~\cite[pp.~117--121]{C}.

\begin{con}[Chen]\label{conj1}  For all $\ts n>1$, we have
\begin{equation}\label{eq:1}
\frac{p(n-1)}{p(n)}\left(1 + \frac{1}{n}\right) \,>\, \frac{p(n)}{p(n+1)}\,.\end{equation}
\end{con}

\begin{con}[Chen] \label{conj2} For all \ts $n>m>1$, we have
\begin{equation}\label{eq:2}
p(n)^2 \. - \. p(n-m)\. p(n+m) \, \ge \, 0 \..
\end{equation}
\end{con}

Note that inequality~\eqref{eq:1} goes in the opposite
direction to log-concavity.  We fully establish Conjecture~\ref{conj1} in Theorem~\ref{t:1} and show
that the $(1+1/n)$ term can be further sharpened to $(1+O(n^{-3/2}))$, see
Theorem~\ref{conjecture 1}.  The inequality~\eqref{eq:2} is sometimes called
\emph{strong log-concavity} when extended to all~$m\ge 1$;
we prove Conjecture~\ref{conj2} in Theorem~\ref{t:3}.

\smallskip

The rest of the paper is structured as follows.  We first prove Theorem~\ref{log concave theorem}
in Section~\ref{proof section}. In a short but technical Section~\ref{sec:decay} we show that
the proof of Theorem~\ref{log concave theorem} gives sharp results on the asymptotic rate
of decay of log-concavity.  These are then used to prove of Chen's conjectures
(sections~\ref{s:chen} and~\ref{s:strong}), and Sun's strongly related conjecture 
(Section~\ref{sec:Sun}). We conclude with final remarks in Section~\ref{sec:final}.

\bigskip

\section{Proof of Theorem~\ref{log concave theorem}} \label{proof section}

\noindent
The log-concavity is equivalent to
\[ \log a_{n-1} \. - \. 2\log a_n \. + \. \log a_{n+1} \. \leq \. 0 \quad \ \text{for all~$n$\.,}\]
which explains the name.

Recall the Rademacher's convergent series for~$p(n)$, based on the Hardy--Ramanujan asymptotic series
\cite{R} (see also e.g.~\cite{An,Ha}). Let $\ts \mu(n) = \frac{\pi}{6}\ts \sqrt{24\ts n-1}$, which we abbreviate to~$\mu$ whenever
convenient.  We have
\begin{equation}\label{HR p(n)}
p(n) \, = \, \frac{\sqrt{12}}{24\ts n-1} \sum_{k=1}^N A_k^{\ast}(n)\left[\left( 1 - \frac{k}{\mu}\right) e^{\mu/k}+\left( 1 + \frac{k}{\mu}\right) e^{-\mu/k}\right] \. + \. R_2(n,N)\ts,
\end{equation}
where $A_k^\ast(n)$ is a complicated arithmetic function that will not be needed in its full generality for our analysis; see \cite{L1, L2} for the complete definition.  We shall only need that $A_1^\ast(n) = 1$ and $A_2^\ast(n) = (-1)^n/\sqrt{2}$ for all positive $n$.  Recall also Lehmer's error bound \cite{L1}
\begin{equation}
|R_2(n,N)| < \frac{\pi^2 N^{-2/3}}{\sqrt{3}} \left[ \left(\frac{N}{\mu}\right)^3 \sinh\frac{\mu}{N} + \frac{1}{6} - \left(\frac{N}{\mu}\right)^2\right],
\end{equation}
valid for all positive $n$ and~$N$.  
  Define for all $n\geq 1$
\begin{equation}
T(n) :=   \frac{\sqrt{12}}{24\ts n-1} \left[\left( 1 - \frac{1}{\mu}\right) e^{\mu} + \frac{(-1)^n}{\sqrt{2}}e^{\mu/2}\right].
\end{equation}
  The function $T(n)$ contains the three largest individual contributions to the sum,
  and is a refinement of the right-hand side of Equation~(\ref{HR asymptotics}).
  Define also the remainder term
\begin{equation}
R(n) :=  \frac{d}{\mu^2} \left[ \left( 1 + \frac{1}{\mu}\right) e^{-\mu} - \frac{(-1)^n}{\sqrt{2}}\frac{2}{\mu}e^{\mu/2} + \frac{(-1)^n}{\sqrt{2}}\left(1 + \frac{2}{\mu}\right)e^{-\mu/2}\right]  + R_2(n,2),
\end{equation}
where $d =  \frac{\pi^2}{6 \sqrt{3}}$\ts.  We may rewrite Equation~\eqref{HR p(n)} as
\begin{equation}
p(n) \. = \. T(n) \. + \. R(n)\ts.
\end{equation}



\begin{lemma}\label{lemma 0}{
Suppose $f(x)$ is a positive, increasing function with two continuous derivatives
for all $x> 0$, and that $f'(x)>0$ and decreasing, and $f''(x)<0$ is increasing for all $x>0$.   Then
\[  f''(x-1) < f(x+1)-2f(x) + f(x-1) < f''(x+1) \quad \text{for all \ $x>1$}\ts. \]
}\end{lemma}


\begin{lemma}\label{lemma T(n)} Let
\[T_1(n) :=2 \ts \log \ts T(n) \. - \. \log \ts T(n-1) \. - \. \log \ts T(n+1).\]
Then, for all \ts $n \geq 50$, we have
\ignore{\begin{equation}
\frac{24\ts \pi}{(24(n+1)-1)^{3/2}} \. - \. \frac{4 \times 288}{(24(n-1)-1)^2} <  T_1(n) \, < \, \frac{24\pi}{(24(n-1)-1)^{3/2}} \. + \. \frac{4 \times 288}{(24(n+1)-1)^2} \..
\end{equation}}
\begin{equation}\label{T bound}
\frac{24\ts \pi}{(24(n+1)-1)^{3/2}} \. - \. \frac{3}{n^2} <  T_1(n) \, < \, \frac{24\ts \pi}{(24(n-1)-1)^{3/2}} + e^{-C \ts \sqrt{n} \ts /\ts 10},
\end{equation}
where $C = \pi\ts \sqrt{\frac{2}{3}}$.
\end{lemma}

\begin{proof}
We prove a stronger result, that for all $n\geq 2$, 
    we have
\[\log T(n) \, = \, \log d \. +\. \mu \. + \. \log(\mu-1) \. - \. 3 \log \mu \.
+ \. \log\left(1 + \frac{(-1)^n e^{-\mu/2} \mu^3}{d \sqrt{2} (\mu-1)}\right).\]
  By Lemma \ref{lemma 0},
\begin{equation}\label{lower bound} T_1(n) >  -\mu''(n+1) - (\log(\mu(n+1)-1))'' + 3 \log(\mu(n-1))''+ G(n),\end{equation}
where $G(n)$ is the expression with the more complicated logs.  
  This simplifies to
\[ T_1(n) \, > \, \frac{24\pi}{(24(n+1)-1)^{3/2}} \, + \, \frac{288\pi(-3+\pi\sqrt{24(n+1)-1})}{(24(n+1)-1)^{3/2}(-6+\pi\sqrt{24(n+1)-1})^2} \]
\[ - \ \frac{3 \times 288}{(24(n-1)-1)^2} \, + \, G(n)\.,\] 
valid for all $n\geq 2$.  Let $\ts x_n =  \frac{(-1)^n e^{-\mu/2} \mu^3}{d \sqrt{2} (\mu-1)}$.
Using $\ts \log(1-|x|) \leq -\log(1-|x|)$ for all $|x|<1$, and $\log(1-x) \geq -x / (1-x)$ for $0<x<1$, we have
\[ G(n) \. > \. 4 \ts \log \left( 1-|x_{n+1}|\right) \. > \.
\frac{-4|x_{n+1}|}{1-|x_{n+1}|} \. > \. -e^{-C \ts \sqrt{n} \ts / \ts 10} \ts , \quad \text{for all $\ts n \geq 50$}\ts.\]
For the upper bound, similarly note that
\[ G(n) \. < \. -4 \log( 1 - |x_{n-1}|) \. < \. \frac{4|x_{n-1}|}{1-|x_{n-1}|} \. < \. e^{-C \ts \sqrt{n} \ts / \ts 10},  \quad \text{for all $\ts n \geq 50$}\ts, \]
as desired.
\end{proof}

\begin{lemma}\label{log rate} Let $y_n = |R(n)|/T(n)$.  Then
\begin{equation}\label{log equation} \log \left[ \frac{(1-y_n)(1-y_n)}{(1+y_{n-1})(1+y_{n+1})}\right]
\. > \. -e^{-C \ts \sqrt{n} \ts / \ts 10}, \quad \text{for all $\ts n\geq 10$}\ts.
\end{equation}
\end{lemma}
\begin{proof}
Simply note that for all \ts $n\geq 2$, we have
\begin{equation}
|R(n)| \. < \. 1 + \frac{16}{\mu^3} \ts e^{\mu/2}\.,
\end{equation}
and
\begin{equation}
0 \. < \. \frac{|R(n)|}{T(n)} \. < \. e^{-C \ts \sqrt{n} \ts / \ts 10}\..
\end{equation}
The result easily follows. \end{proof}

\begin{prop}\label{almost proof prop}
Let $\ts p_2(n) = 2\log p(n) - \log p(n-1) - \log p(n+1)$.  Then,
for all $\ts n \geq 2600$, we have
\begin{equation}\label{p bound}
\frac{1}{(24\ts n)^{3/2}} \. < \. p_2(n) \. < \. \frac{2}{n^{3/2}}\..
\end{equation}
\end{prop}

\begin{proof}
We start by writing a double strict inequality for~$p(n)$,
\[ T(n)\ts \left( 1 - \frac{|R(n)|}{T(n)}\right) \. < \. p(n) \. < \. T(n)\ts \left( 1 + \frac{|R(n)|}{T(n)}\right),
\quad \text{for all \ $n\geq 1$}\ts .\]
We then take logs and apply the Lemmas, and obtain
\begin{equation}\label{more explicit p bound}
\frac{24\ts\pi}{(24(n+1)-1)^{3/2}} \. - \. \frac{3}{n^2} \. - \. 2\ts e^{-C \ts \sqrt{n} \ts / \ts 10}
\, < \, p_2(n) \, < \, \frac{24\ts\pi}{(24(n-1)-1)^{3/2}} \. + \. 2\ts e^{-C \ts  \sqrt{n}\ts / \ts 10}\..
\end{equation}
Now the result easily follows by a direct calculation.
\end{proof}

\bigskip

\section{Rate of decay}\label{sec:decay}

\noindent
The rate $O(n^{3/2})$ in Lemma \ref{lemma T(n)} is asymptotically sharp.
Based on the asymptotics of~$T(n)$, we have the following corollary.

\begin{cor}\label{rate theorem}
Let $D(n) = 2 \log p(n) - \log p(n-1) - \log p(n+1)$.   We have
\[\lim_{n\to\infty} \. D(n) \. \frac{4\ts n^{3/2}}{C} \, = \. 1, \quad \text{where} \ \ C = \pi \sqrt{\frac{2}{3}}\..\]
\end{cor}

A plot of exact values for $n$ up to 2000 is given in Figure \ref{rate figure}.  This asymptotic rate comes from the Taylor Series of $\mu \sim 2c\sqrt{n}$ starting with the second derivative term.
\begin{figure}
\includegraphics[height=5cm]{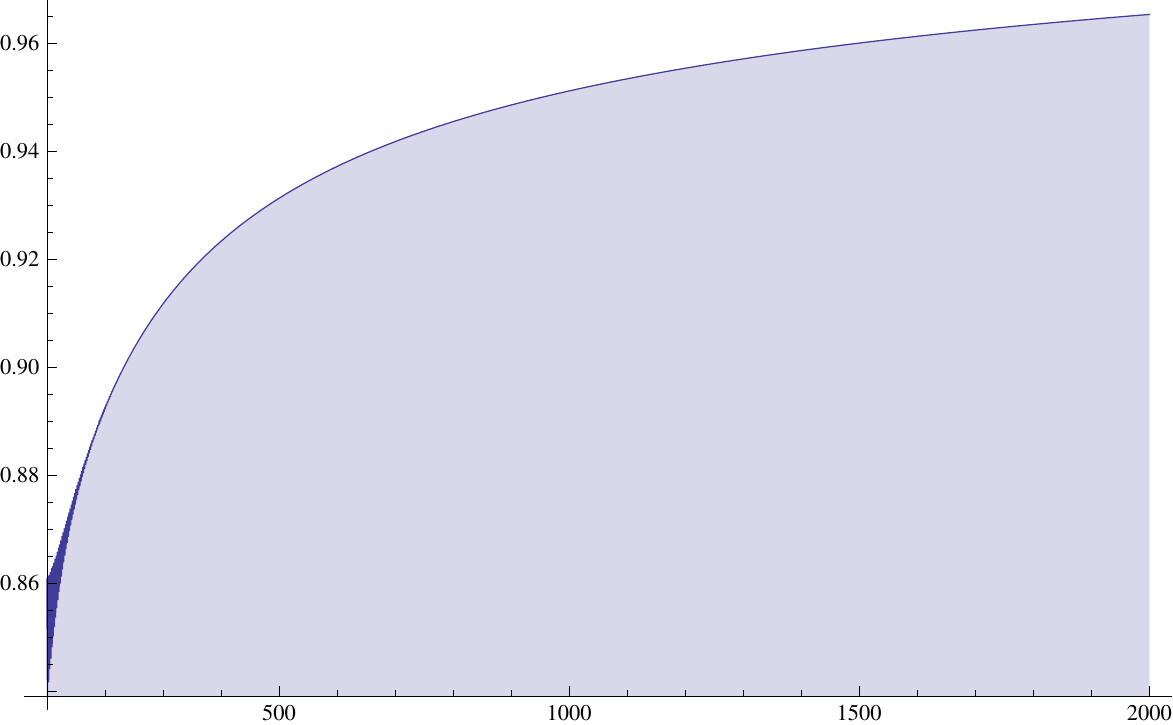}
\caption{A plot of \. $n^{\frac32}\ts D(n)/\bigl(\frac{\pi}{\sqrt{24}}\bigr)$,  for \ts $n=1,\ldots, 2000$.}
\label{rate figure}\end{figure}

A refinement of the rate in Theorem \ref{rate theorem} to higher order terms follows easily as long as we also include the higher order terms of $-2\log(\mu)$.  Since $(2\log(\mu))'' \sim O(n^{-2})$, these terms did not appear in the first-order asymptotic expansion, but higher order expansions must take them into account.  


Let
\[h_1(x) \. = \. \frac{4}{C\ts x^{3/2}}\quad \ \, \text{and}  \quad h_2(x) \. =  \. \frac{-288}{(-1+24 x)^2}\.. \]
Fix an integer $k\ge 0$.  Denote by $L_k^+(n)$ the first $k$ even terms in the Taylor series expansion of $h_1(n+1)+h_2(n+1)$ evaluated at $x_0 = n$, $k=0,1,\ldots$.  Similarly, denote by $L_k^-(n)$ the first $k$ even terms in the Taylor series expansion of $h_1(n+1)$ evaluated at $x_0 = n$, plus the first $k-1$ even terms in the Taylor series expansion of $h_2(n+1)$ evaluated at $x_0 = n$, $k=1,2,\ldots$, with $L_0^-(n) = h_1(n)$.

\begin{cor}   For every integer \ts $k\ge 0$, we have
\[ D(n) \. \sim \. L_k^+(n)\ts, \quad D(n) \. \sim \. L_k^-(n)\quad \text{as \,$n\to \infty$}\ts. \]
\end{cor}


\bigskip

\section{Reverse inequality}\label{s:chen}

\noindent
In this section, we establish and further strengthen Conjecture~\ref{conj1}.

\begin{thm}[formerly Chen's Conjecture~\ref{conj1}]\label{t:1}  For all \ts $n>1$, we have
\begin{equation}\label{eq:rev}
\frac{p(n-1)}{p(n)}\left(1 + \frac{1}{n}\right) \,>\, \frac{p(n)}{p(n+1)}\,.\end{equation}
\end{thm}

\begin{proof}
We can rewrite Equation~\eqref{eq:rev}  as
\[ p_2(n) \. < \. \log\left(1 + \frac{1}{n}\right). \]
By Proposition \ref{almost proof prop}, for all \. $n\geq 2600$ we have
\[ p_2(n) \. < \. \frac{2}{n^{3/2}} \. < \. \frac{1}{n+1} \. < \.
\log\left(1 + \frac{1}{n}\right). \]
We then check numerically in {\sc Mathematica} that
Equation~\eqref{conjecture 1 eq} holds for all positive $n< 2600$.
\end{proof}

We next strengthen Theorem~\ref{conjecture 1} to the correct rate.

\begin{thm}\label{conjecture 1} For all \ts $n> 6$, we have
\begin{equation}\label{conjecture 1 eq}
\frac{p(n-1)}{p(n)}\left(1 + \frac{240}{(24n)^{3/2}}\right) \, > \, \frac{p(n)}{p(n+1)}\..\end{equation}
\end{thm}
\begin{proof}
The theorem follows similarly from the proof above, since for all \ts $n \geq 22$, we have
\[p_2(n) \. < \. \frac{2}{n^{3/2}} \, < \, \frac{240}{(24n)^{3/2} + 240} \, < \, \log\left( 1 + \frac{240}{(24n)^{3/2}}\right).\]
\end{proof}

\begin{conjecture}\label{conj3} For all \ts $n\geq 45$, we have
\[\frac{p(n-1)}{p(n)}\left(1 + \frac{\pi}{\sqrt{24} \. n^{3/2}}\right) > \frac{p(n)}{p(n+1)}\,.\]
\end{conjecture}

We have checked this conjecture in {\sc Mathematica} for values of $n \leq 8000$, with violations occurring every even number less than~45.  Lemma~\ref{lemma T(n)} comes close to a proof, but one would have to check carefully the error involved in replacing the term $24(n-1)-1$ with $24\ts n$, as well as the exponential contribution.\footnote{Added in proof: Conjecture 4.3 was proved by
W.Y.C. Chen et al. in~\cite{CWX}.}

\bigskip

\section{Strong log-concavity}\label{s:strong}

\noindent
In this section, we establish Conjecture~\ref{conj2} on strong log-concavity of the partition function.

\begin{thm} [formerly Chen's Conjecture~\ref{conj2}]
\label{t:3}
For all \ts  $n >m>1$, we have
\begin{equation}
p(n)^2 \. - \. p(n-m)\ts p(n+m) \. > \. 0\ts.
\end{equation}
\end{thm}

\begin{proof}
First, recall that log-concavity implies strong log-concavity (see e.g.~\cite{Sa})
\[ a_k \ts a_\ell \. \leq \. a_{k+i}\ts a_{\ell-i}\., \]
for all $0 \leq k \leq \ell \leq n$ and $0 \leq i \leq k - \ell$.
Take $k=n-m$, $\ell = n+m$, and~$i = m$, to obtain:
\begin{equation}
p(n)^2 \. - \. p(n-m)\ts p(n+m) \. > \. 0 \quad \text{for all} \ \ n > 25+m\ts.
\end{equation}
To resolve the remaining cases, we use the following classical bounds which hold for all $m\ge 1$~:
$$e^{C \sqrt{m}} \, > \, p(m) \, > \,
\frac{e^{2\sqrt{m}}}{2\pi m \ts e^{1/6m}}\,.
$$
The first inequality was given in~\cite{E}, and the second is from Section~2 of~\cite{HR}.  Next,
observe that 
\begin{align*}
p(n-m) < p(25) <2000 & \quad \text{for all \, $1 \leq n-m \leq 25$\ts.}
\end{align*}
Hence, for all $1 \leq n-m \leq 25$, we have
\[ p(n)^2 \.\geq \. p(m+1)^2 \. \geq\. p(25)\ts p(25+m) \. \geq \. p(n-m)\ts p(n+m)\.,\]
and thus it suffices to prove the middle inequality for all $m \geq 2$.  But this is equivalent to
\[ 2 \log p(m+1) - \log p(25) - \log p(25+m) \. > \. 0\ts, \]
and by the inequalities above it is sufficient to establish for which values of $m$ the following holds:
\[ 4 \sqrt{m+1} - 2 \log(m+1) - \frac{1}{3(m+1)} - 2 \log(2\pi) - \log 2000 - C \sqrt{m+25} \. > \. 0\ts.\]
This is easily seen to hold for $m \geq 300$, and a calculation in {\sc Mathematica} easily establishes the remaining cases.
\end{proof}

\bigskip

\section{Sun's conjecture}\label{sec:Sun}
In \cite{Sun}, a similar conjecture is stated for $q(n) := p(n) / n$.
We prove it here.  

\begin{thm}[formerly Sun's Conjecture]
The sequence  $\{q(n)\}_{n \geq 31}$ is log-concave.  
\end{thm}
\begin{proof}
Let $q_2(n) = 2\log q(n) - \log q(n-1) - \log q(n+1)$.  Then
\[ \log q_2(n) = \log p_2(n) - \log\left(\frac{n^2}{(n-1)(n+1)}\right). \]
By Lemma~\ref{lemma 0} with $f(x) = \log x$, we have
\[ \frac{1}{(n+1)^2} < \log\left(\frac{n^2}{(n-1)(n+1)}\right) < \frac{1}{(n-1)^2}. \]
We simply need to modify Equation~\ref{more explicit p bound} to include this new error, i.e.,
\begin{equation}\label{q bound}
\frac{24\ts\pi}{(24(n+1)-1)^{3/2}} \. - \. \frac{3}{n^2} \. - \. 2\ts e^{-C \ts \sqrt{n} \ts / \ts 10}
\, -\frac{1}{(n+1)^2}\, < \, q_2(n) \, < \, \frac{24\ts\pi}{(24(n-1)-1)^{3/2}} \. + \. 2\ts e^{-C \ts  \sqrt{n}\ts / \ts 10}\..
\end{equation}
The direct calculation is the same as in the proof of Lemma~\ref{almost proof prop}, which establishes 
the log-concavity of $q(n)$ for $n\geq 2600$.  Since Sun checked numerically the log-concavity 
of $q(n)$ for $n \leq 10^5$, see~\cite{Sun}, this completes the proof.
\end{proof}

\bigskip
\section{Final remarks}\label{sec:final}

\subsection{The log-concavity conjecture}\label{ss:fin-folk}
Let us mention that although the conjecture behind Theorem~\ref{log concave theorem}
seems to be folklore, we were unable to locate it in the literature.
The problem was revived by Chen~\cite{C}, who checked it for all $\ts n\le 8000$,
by Janoski (see the discussion below in~\S\ref{ss:fin-jan}), and most recently by Burde on \ts
{\it MathOverflow}.\footnote{See \ts {\tt http://tinyurl.com/kkc6fwf}\ts.}
After this paper was finished, we learned that it was also independently conjectured
by Andrews and Hirschhorn.\footnote{Personal communication.}  Most recently, a 
different but somewhat related investigation was made in~\cite{BO}.

\subsection{Historical background}\label{ss:fin-hist}
The classical Hardy--Ramanujan formula is
\begin{equation}\label{HR full}
p(n) \. = \. \frac{\sqrt{12}}{24\ts n-1} \sum_{k=1}^N A_k^{\ast}(n)\left[\left( 1 - \frac{k}{\mu}\right) e^{\mu/k}\right] \. + \. R_1(n,N)\ts.
\end{equation}
They derived the explicit forms of the $A_k^{\ast}(n)$ for $k=1,\ldots, 18$ and all $n$ as a linear combination of cosines~\cite{HR}.  This allowed them to compute expansions for various small values of $n$ that were checked by MacMahon.  In particular, two terms of their series give
\begin{equation}
p(n) \, \sim \, \frac{1}{2\pi \sqrt{2}} \.\frac{d}{dn} \left(\frac{e^{C \lambda_n}}{\lambda_n} \right) \. + \. \frac{(-1)^n}{2\pi} \. \frac{d}{dn}\left(\frac{e^{\frac{1}{2}C \lambda_n}}{\lambda_n}\right) \. + \. O(e^{D \sqrt{n}}),
\end{equation}
where
\[  \qquad C \. = \. \pi \ts \sqrt{\frac{2}{3}}\ , \qquad \lambda_n \. = \. \sqrt{24\ts n-1}
\quad \ \ \text{and any} \ \quad D \ts > \. \frac{C}{3} \..\]
This series was also extended to include $N = O(\sqrt{n})$ terms with an error of $O(n^{-1/4})$.  However,
it was later shown by Lehmer~\cite{L0} that this series diverges if extended to infinity.
The first term of the series easily suffices to establish log-concavity for sufficiently
large~$n$ (see~\S\ref{ss:fin-hist}), but there is no natural way to make these bounds explicit.
Rademacher modified the formula and produced a convergent series in~\cite{R}.
Most recently, this series used by Janoski~\cite[$\S 1.2$]{J}, to claim the proof of Theorem~1.
Unfortunately, this proof has a serious logical flaw (see below).

This relation has since been interpreted in the popular literature as
\[ p(n) = \frac{e^{C\ts \sqrt{n}}}{4 \sqrt{3}\, n} \left( 1 + O(n^{-1/2})\right). \]
It follows then, that one cannot base any heuristic interpretation of log-concavity
on this simplest form, as it has an error that is larger than the asymptotic
rate of decay given by Theorem~\ref{rate theorem}.

One can, however, embrace the \emph{entire} first-order term of the Hardy-Ramanujan series, and obtain
\[ p(n) = \frac{e^{C\ts \sqrt{n-\frac{1}{24}}}}{4 \sqrt{3}\, \left(n-\frac{1}{24}\right)}\left(1 - \frac{1}{C \sqrt{n-\frac{1}{24}}}\right) \left( 1 + O(e^{-\mu/2})\right). \]
The relative error term above decreases exponentially, and hence it easily follows that $\{p(n)\}$ is \emph{asymptotically log-concave},
i.e., that~$\{p(n)\}$ is log-concave for all $n>n_0$, for some \emph{unspecifiable}~$n_0$.\footnote{See e.g.,~Speyer's calculation in \. {\tt http://tinyurl.com/nyq2zrn} }  Without a concrete error bound in the form of a strict inequality, one cannot in general obtain the result.

The \emph{Rademacher series} for $p(n)$ is a convergent series of the form
\[ p(n) \, = \, \frac{\sqrt{12}}{24\ts n-1} \, \sum_{k=1}^\infty \. A_k^{\ast}(n)\ts \left[\left( 1 - \frac{k}{\mu}\right) e^{\mu/k} \ts +\ts \left( 1 + \frac{k}{\mu}\right) e^{-\mu/k}\right]\ts. \]
This series was shown to converge by Rademacher.  Equation \eqref{HR p(n)} is the truncation of this series to $N$ terms, and Rademacher provided a \emph{strict} inequality on the absolute value of the error term $R_2(n,N)$ for all $n$ and $N$.  This was later improved by Lehmer, and we have used his estimates even though in principle Rademacher's estimates would have been sufficient.

An important part of the proof of Theorem~\ref{log concave theorem} is the form of the error.  
 Rademacher's upper bound can be written for some positive constant $c'$ as
\[|R_2(n,N)| \. < \. c' e^{\mu/N} \frac{\sqrt{N}}{\mu} \quad \text{for all \. $n\ts, \ts N\geq 1$\ts.}\]
Lehmer's upper bound can be written for some other positive constant $c''$ as
\[|R_2(n,N)| \. < \. c'' e^{\mu/N}\frac{N^3}{\mu^3} \quad \text{for all \. $n\ts, \ts N\geq 1$\ts.}\]
Thus, neither Rademacher's nor Lehmer's upper bounds for $N=1$ provide a sharp enough bound to have a relative error $o(n^{-3/2})$.  In particular, the exponent $\mu/N$ implies that we must take $N \geq 2$ in order to obtain an exponentially decaying error term relative to the first term.  This demonstrates that one cannot obtain even a proof of asymptotic log-concavity using this error analysis with $N=1$.

\subsection{Janoski's thesis}\label{ss:fin-jan}
Recently, Janoski in her thesis~\cite{J} claimed to have a complete proof of log-concavity, for all~$n>25$.
Unfortunately, there is a serious technical error early in the proof, invalidating the argument.
Specifically, Janoski writes on page~9 the following inequality:
$$\frac{1}{\pi \sqrt{2}} \. \sum_{k=3}^{C\sqrt{t}} \. A_k(n) \sqrt{k} \. d(n,k) \, > \, C \sqrt{t} \ts \frac{\sqrt{3}}{\pi \sqrt{2}} \. A_3(n) \. d(n,3),$$
where
$$t \. = \. n\ts - \ts \frac{1}{24}\., \qquad  A_k(n) \. = \. A_k^\ast(n)/\sqrt{k}\., \quad \text{and} \quad d(n,k) \. = \.
\left(\frac{\sinh C\sqrt{t}/k}{\sqrt{t}} \right)'.
$$
It seems, the author assumes that the functions $A_k(n)$ are positive and monotonic in~$k$.  We have checked this inequality for various small values of~$n$, and have found several counterexamples, such as $n$ equal to $27, 36, 87$ or~$744$.  Furthermore, the calculation bounding the absolute error $R_2(n,N)$ is incorrect. The rest of the analysis in~\cite{J} is based on bounding from below this faux lower bound term, which has a relative error $O(e^{-2\mu/3})$ (though not in the form of a strict inequality!), and so naturally the numerical results would appear to confirm the entire result.

\subsection{Computer calculations}
All of the calculations for small values of~$n$ mentioned in the proof above are trivial, if tedious.
The calculations were done with {\sc Mathematica}, which has a built-in partition function.
%
%
Let us note, however, that strictly speaking, the use of computer is not necessary to follow
and verify the proofs of the Lemmas, as there are comprehensive tables of the partition function,
see e.g.,~\cite{G}.

\subsection{Combinatorial proof}
It is natural to ask for a direct combinatorial proof of
Theorem~\ref{log concave theorem}, e.g.,~in the style of proofs in~\cite{Bo,HS,Kr,Sa}.
Unfortunately, the previous attempts and the fact that log-concavity fails for
small~$n$, are very discouraging.   It seems, even the simplest results for
the partition function are difficult to prove directly (cf.~\cite{A}).

Note, however, that for related sequences such as \emph{convex compositions}
(or \emph{stacks}), the asymptotic formulas are similar to~\eqref{HR asymptotics},
and thus not sharp enough for the analysis as in this paper (see~\cite{Br,Wr}).
A combinatorial proof is the best current hope for proving asymptotic
log-concavity of such sequences.

\subsection{Another Chen's conjecture}
Note that Theorem~\ref{t:3} says that the sequences $\{p(m\ts n + a)\}$ are
log-concave, for all fixed $m > a \ge 0$. In~\cite{C}, Chen also makes
a related conjecture, that for all $a>b$ we have
\[p(a\ts n)^2 \. - \. p(a\ts n - b\ts n)\. p(a\ts n + b\ts n) \. > \. 0\quad
\text{whenever \ $a\ts n, \. b\ts n \ts \in \mathbb{N}$\ts.} \]
Of course, this is a weak version of Conjecture~\ref{conj2}.
Now that Theorem~\ref{t:3} is proved, it also follows.

\subsection{Another Sun's conjecture}
Let $r(n) := \sqrt[n]{p(n)/n}$.  In~\cite{Sun} (see also~\cite{Sun1}), 
Sun conjectures that $\{r(n)\}$ 
is log-concave for $n\ge 60$.  Unfortunately, the calculation for $r(n)$ 
does not follow immediately from our lemmas, since there is no simple lemma 
analogous to Lemma~\ref{lemma 0}.  However, it is easy to see that 
such result only requires a more careful bounding of elementary functions;
we hope the reader will continue this investigation.  Let us mention also
that in the same paper, Sun makes similar conjectures for partitions 
into distinct part, supported by numerical evidence.

\vskip.8cm

\subsection*{Acknowledgements}
The authors would like to thank William Chen, Christian Krattenthaler,
Karl Mahlburg, Bruce Rothschild, Bruce Sagan for helpful remarks and suggestions.
The authors are grateful to Richard Arratia for introducing us to the problem,
to David Moews for bringing the work of Lehmer to our attention, and to
Janine Janoski for informing us about the status of~\cite{J}.
The second author was partially supported by the~NSF.

\vskip.9cm


{\footnotesize

\end{document}